\documentclass [11pt,reqno]{amsart}
\usepackage {amsmath,amssymb,verbatim,geometry}
\usepackage[all]{xy}
\newif\ifpdf
\ifpdf
\usepackage[pdftex]{hyperref}
\else
\fi

\numberwithin{equation}{section}       

\newtheorem{prop} {Proposition} [section]
\newtheorem{thm}[prop] {Theorem} 
\newtheorem{quest} {Question} 
\newtheorem{defi}[prop] {Definition}
\newtheorem{lem}[prop] {Lemma}
\newtheorem{cor}[prop]{Corollary}
\newtheorem{prop-def}[prop]{Proposition-Definition}

\newtheorem{exa}[prop]{Example}

\theoremstyle{remark}
\newtheorem{rem}[prop]{Remark}

\newcommand{\C}{{\mathbb{C}}}

\newcommand{\R}{{\mathbb{R}}}
\newcommand{\Rn}{{\mathbb R^n}}

\newcommand{\D}{\Delta}

\title[Regularity of geodesics]{Regularity of geodesics in the spaces of convex and plurisubharmonic functions II}

\date{\today}

\author{Soufian Abja, S\l awomir Dinew}

\address{Mathematisches Institut, Universitat Münster, 48149 Münster, Germany}
\email{sabja@uni-muenster.de}
\address{Institute of Mathematics, Jagiellonian University, ul Lojasiewicza 6, 30-348 Krak\'ow, Poland}
\email{Slawomir.Dinew@im.uj.edu.pl }

\begin{document}
	\begin{abstract}
		In this note we continue our investigation of geodesics in the space of convex and plurisubharmonic functions. We show optimal regularity for geodesics joining two smooth strictly convex functions. We also investigate the regularity theory in the $C^{1,\alpha}$ realm. Finally we discuss the regularity of geodesics joining two toric strictly plurisubharmonic functions. 
	\end{abstract}
	
	\maketitle

	\section*{Introduction} This is a follow-up note of our previous manuscript \cite{AD21}. Recall that given a smoothly bounded strictly convex domain $U\subset\Rn$ we consider the set $\mathcal S$ of smooth strictly convex (up to the boundary) functions defined by
	$$\mathcal{S}:=\{u\in C^{\infty}(\overline{U})| D^2u>0\ {\rm on}\ \overline{U} , u=0\; {\rm on}\ \partial U\}.$$
	
	Next, we think of $\mathcal S$ as an infinite dimensional Riemannian manifold with tangent space at each $u\in\mathcal S$ identified with $C^{\infty}(\bar{U},\mathbb R)$ vanishing on the boundary and metric given by
	$$\langle f_1,f_2\rangle_u:=\int_{U}f_1f_2det(D^2u)$$
	for any $u\in\mathcal{S}\ ,f_1,\ f_2 \in T_{u}\mathcal{S}.$
	
	Then, as simple computation shows (see \cite{AD21}) the {\it geodesic equation} for the geodesic $u:[0,1]\longmapsto \mathcal S$ joining $u_0,u_1\in\mathcal{S}$ reads
	
	\begin{equation}\label{ss}
		\begin{cases}
			u\in C(\overline{U\times(0,1)})\ {\rm and\ convex\ in}\ U\times(0,1);\\
			\det(D^2 _{x,t}u)=0\;{\rm in}\ \;\; U\times (0,1);\\
			u=u_0\;\; {\rm in}\ \;\; U\times \{0\};\\
			u=u_1\;\; {\rm in}\ \;\; U\times \{1\};\\
			u=0\;\; {\rm in}\ \;\; \partial U\times (0,1),
		\end{cases}
	\end{equation}
	where $D^2_{x,t}u$ is the Hessian of $u$, thought as a function on the product $U\times (0,1)$ with respect to both variables $(x,t)$. Such a construction is the real analogue of Mabuchi's space of K\"ahler metrics on a compact K\"ahler manifold (see \cite{Mab87}). The link with the Monge-Amp\`ere equation in the K\"ahler case was observed by Semmes \cite{Sem92} and Donaldson \cite{D99}.
	
	The equation (\ref{ss}) is a special case of a Dirichlet problem for the homogeneous Monge-Amp\`ere equation for which a vast literature exists. In particular in \cite{CNS86} it was shown that smooth data and smooth boundary values yield $C^{1,1}$- regular solutions which is optimal in general. We refer also to \cite{LW15}, where the homogeneous problem in an unbounded strip domain is studied.
	
	The main point in (\ref{ss}) is that the problem is posed over a domain {\it with corners} and merely Lipschitz boundary regularity is available. Hence the theory from  \cite{CNS86} does not apply. On the bright side the boundary data is of very special shape.
	
	The main result of \cite{AD21} can be restated as follows:
	
	\begin{thm}
Let the convex function $u$ solve the Dirichlet problem (\ref{ss}) as above. Then u is globally Lipschitz and $C_{loc}^{1,1}$ regular. A blow-up of the $C^{1,1}$ norm is only possible at the corners $\partial U\times\lbrace0,1\rbrace$. Furthermore the solution is $C^{\infty}$ smooth if and only if the gradient images of $u_0$ and $u_1$ agree i.e.

$$\partial u_0(U)=\partial u_1(U).$$ 
	\end{thm}

In the current note we investigate what happens if we modify somewhat the above setting. First of all in the case of the space $\mathcal S$ the assumed strict convexity is stronger than the usual one. Typically one works with the local strict convexity i.e. with the space
\begin{equation}\label{newS}
\tilde{\mathcal{S}}:=\{u\in C^{\infty}({U})\cap C(\overline{U})| D^2u>0\ {\rm on}\ {U}, \ ||u||_{C^2}<\infty, u=0\; {\rm on}\ \partial U\}.	
\end{equation}	
The reason we have worked with $\mathcal S$ in \cite{AD21} is that various implicit function arguments were applied and, unless we control the constants involved, all estimates seemed to break close to the boundary. This is also the reason for imposing finite $C^2$ norm in the current note. Thus, roughly speaking, an element in $\tilde{\mathcal{S}}$ can have minimal eigenvalue tending to zero as the base point goes to the boundary, but the largest eigenvalue is bounded from above.

Note however that the construction described above can also be repeated for $\tilde{\mathcal{S}}$ with the same geodesic equation - the only necessary change is that now  the tangent space $T_{u}\tilde{\mathcal{S}}$ should be the space of {\it test functions} $C^{\infty}_{0}({U},\mathbb R)$. Hence a question appears\footnote{Both authors thank Eleonora Di Nezza for discussions on the topic.} whether the findings from \cite{AD21} remain true in the space $\tilde{\mathcal{S}}$. Our first observation is that this is the case.

\begin{thm}\label{main}
	Let the convex function $u$ solve the Dirichlet problem (\ref{ss}) as above with $u_0,u_1\in \tilde{\mathcal S}$. Then $u$ is globally Lipschitz and $C^{1,1}_{loc}$ regular. It is globally $C^{1,1}$ smooth in $U\times(0,1)$ if and only if the gradient images of $u_0$ and $u_1$ agree.
\end{thm}

A second issue we deal in this note is the optimal regularity for the solution of the geodesic equation when {\it less} smoothness is initially assumed. To this end we need some definitions.

For any $y\in U$ let $l_y(x)$ denotes an affine support function of $u$ at $x$. i.e. $u(x)\geq l_y(x)$ throughout $U$ with an equality at $y$. If $u$ is differentiable at $y$ then obviously $l_y(x)$ is unique and is given by $l_y(x):=u(y)+Du(y)\cdot (x-y)$.

\begin{defi}
Let $U$ be a convex domain in $\Rn$ and $u$ be a convex function on $U$. $u$ is said to be strictly convex at $y\in U$ if for some affine support function $l_y$ we have $$\lbrace u(x)=l_y(x)\rbrace=\lbrace y\rbrace.$$ $u$ is said to be strictly convex if it is strictly convex at every point $y\in U$.	
\end{defi}

Below we describe the space of strictly convex $C^{1,\alpha}$ regular functions with Dirichlet boundary data.
\begin{defi}
Let $U$ be a convex domain in $\Rn$ and $\alpha\in (0,1]$ be a constant. Then the space  $\tilde{\mathcal S}^{1,\alpha}$ is given by
$$\tilde{\mathcal{S}}^{1,\alpha}:=\{u\in C^{1,\alpha}(U)\cap C(\overline{U})| u\ {\rm strictly\ convex\ on}\ {U} ,\  ||u||_{C^1}<\infty, u=0\; {\rm on}\ \partial U\}.$$	
\end{defi}

The $C^{1,\alpha}$ regularity for the degenerate Monge-Amp\`ere equation was also intensively studied - we refer to \cite{DPF15} and especially the recent paper \cite{CTW22} for an up-to-date results. Once again the geodesic equation is posed in a non-smooth domain and results from these papers cannot be used directly.

Our second result is that, thanks to the specific boundary data, the geodesic equation has the same $C^{1,\alpha}$ regularity theory as in the case of space $\tilde{\mathcal S}$.

\begin{thm}\label{thm2}
Let $u$ solve the Dirichlet problem (\ref{ss}) as above with $u_0,u_1\in  \tilde{\mathcal S}^{1,\alpha}$ for some fixed $\alpha\in (0,1]$. Then $u\in C^{1,\alpha}_{loc}(U\times(0,1))$.	A blow-up of the $C^{1,\alpha}$ norm is only possible at the corners $\partial U\times\lbrace0,1\rbrace$.
\end{thm}

In \cite{Abj19} and \cite{Ras17} the complex analogue of geodesics in domains was considered. There the weak geodesic has to solve a complex Monge-Amp\`ere equation on the product of the domain and a planar annulus $A$ (see the Preliminaries for the details). The analogue of the space $\mathcal S$ is then the space $\mathcal H$ of smooth strictly plurisubharmonic functions vanishing on the boundary. In this situation the problem is much harder and partial results were shown only for special domains (see \cite{Abj19} when the domain is a ball in complex space) or under symmetry assumptions of the boundary data.

In \cite{AD21} we considered the case when the boundary data consists of toric plurisubharmonic functions and the underlying domain is Reinhardt. Thanks to a well known identification (see Lemma \ref{toric} below) in this setting the problem is reduced to (\ref{ss}) on the logarithmic image of the Reinhardt domain. Hence we have shown the following result (see Corollary 0.2 in \cite{AD21}):
\begin{prop}\label{pq}
	Let $\Omega$ be a smoothly bounded strictly pseudoconvex Reinhardt domain in $\mathbb C^n$. Let also $M=\{z\in\mathbb C^n| z_1\cdots z_n=0\}$. Suppose that $\phi$ is a weak geodesic solving the problem (\ref{MAcomplex}). If $\varphi_0,\varphi_1\in\mathcal H$ are toric in the space variables i.e.
	for all $z,z'\in \Omega$ satisfying $|z|=|z'|$ one has $\varphi_j(z)=\varphi_j(z'),\ j=0,1$ then $\phi$ is $C^{1,1}$ away from the corner $\partial\Omega\times\partial A$ 
	and  $(\Omega\cap M)\times A$.
	\end{prop}
The problem with $(\Omega\cap M)\times A$ is that the logarithmic map obviously degenerates along the set $\Omega\cap M$ (unless it is empty). Nevertheless in \cite{AD21} (see Remark 2.4 there) we conjectured that $C^{1,1}$ regularity extends past $(\Omega\cap M)\times A$. Our final result confirms this:
\begin{thm}\label{tc}
If $\Omega,\ \varphi_0,\ \varphi_1$ are as above, then the geodesic $\phi$ is $C^{1,1}$ regular in the $\Omega$ direction past the set $(\Omega\cap M)\times A$. If $\Omega$ is additionally complete Reinhardt then $\phi$ is $C^{1,1}$ in all directions.
\end{thm}

The note is organized as follows. In Section 1 we gather some important facts that shall be used later on. Next, in Section 2 we provide some examples. Section 3 is devoted to the proof of Theorem \ref{main}. In Section 4 we deal with Theorem \ref{thm2}. Finally in Section 5 we handle Theorem \ref{tc}.

{\bf Acknowledgements}. The second named author has been partially supported  by grant no. 2021/41/B/ST1/01632 from the National Science Center, Poland. The first author was partially funded by the German Research Foundation (DFG) under Germany’s Excellence Strategy EXC 2044-390685587 “Mathematics Münster: Dynamics-Geometry-Structure” and by the CRC 1442 “Geometry: Deformations and Rigidity” of the DFG.

\section{preliminaries}

We shall follow the notation used in \cite{AD21}. We also refer to \cite{Gu02} and \cite{Fi17} for background material regarding the Monge-Amp\`ere operator and its basic regularity theory.

{\bf Real geodesics.}

The existence and uniqueness of convex generalized solutions to (\ref{ss}) is standard -see \cite{Gu02,Fi17}. As the boundary data in all the considered settings is (more than) Lipschitz the global Lipschitz regularity of the solution follows from the standard Walsh technique - see \cite{Wal68}.

Recall that (see \cite{Fi17}) the Legendre transform of a continuous function $u: U\ni\mathbb R^n\longmapsto \mathbb R$ is given by
\begin{equation}\label{Leg}
	u^{*}(y)=\sup_{U}\{x.y-u(x)\}.
\end{equation}

Below we list several basic facts about this transformation:
\begin{prop}\label{propLag}
	Let $U\in\mathbb R^n$ be a convex set and $u: U\longmapsto \mathbb R$ be a convex function. Then
\begin{enumerate}
	\item $\{u^*<\infty\}=\partial u(U)$;
	\item if $u$ is smooth and strictly convex on $U$ then so is $u^*$ on $\partial u(U)$;
	\item $u^{**}|_{U}=u$;
	\item If  $u$ is $C^2$ and strictly convex, then
	$$D^2 u^{*}(y)=( D^2 u(x(y)))^{-1},$$
	where $D^2 u$ represents the Hessian matrix of $u$ and $(D^2 u)^{-1}$ its inverse. $x(y)$ is the unique point where the supremum in the definition of $u^*$ is achieved.
\end{enumerate}
\end{prop}
A crucial fact linking the Legendre transform and geodesics is the following nice formula whose proof (even in more general situations) can be found in \cite{Ra17}.
\begin{prop}\label{Ras}
Let $U\in\mathbb R^n$ be a convex set and $u_0,\ u_1$ be two strictly convex functions vanishing on $\partial U$. Then the geodesic $u$ joining them is given by the formula
\begin{equation}\label{fla}
	u^*(y,t)=(1-t)u_0^{*}(y)+tu_1^{*}(y),
\end{equation} 
where $u^*(y,t)$ is the partial Legendre transform done only with respect to the $y$-variable.
 	\end{prop}

{\bf Complex geodesics with toric data.}

As we shall also discuss the geodesic equation for plurisubharmonic functions we include also the necessary definitions.

\begin{defi} Let $\Omega$ be a domain in $\mathbb C^n$. An upper semicontinuous function $\phi$ is called plurisubharmonic, if its restriction to any complex line $L$ is subharmonic or constantly $-\infty$ on each component of $L\cap\Omega$. The space of plurisubharmonic functions will be denoted by $PSH(\Omega)$.
	\end{defi}
These are the complex analogues of convex functions. As we shall deal only with rotationally invariant plurisubharmonic functions in this note it is worth recalling that
$\phi$ is plurisubharmonic and rotationally invariant on its domain of definition iff the function $u(x):=\phi(e^{x_1},\cdots,e^{x_n})$ is convex.

The geodesic equation in this setting reads (see \cite{AD21})
\begin{equation}\label{MAcomplex}
	\begin{cases}
		\phi\in PSH(\Omega\times A)\cap C(\overline{\Omega\times A});\\
		(dd^c_{z,\zeta}\phi)^{n+1}=0 &{\rm in}\  \hbox{$\Omega\times A$}; \\
		\phi=\varphi_0 &{\rm in}\ \hbox{$\Omega\times \{|z|=1\}$}; \\
		\phi=\varphi_1 &{\rm in}\ \hbox{$\Omega\times \{|z|=e\}$}; \\
		\phi=0 &{\rm in}\ \hbox{$\partial \Omega\times A$,}
	\end{cases} 
\end{equation}
where  $A=\{z\in \C| 1<|z|<e\}$ denotes an annulus in $\C$  and $\phi(z,\zeta)=\varphi_t(z)$ with $t=\log|\zeta|$. Finally $(dd^c_{z,\zeta}\cdot)^n$ is the complex Monge-Amp\`ere operator in joint variables $(z,\zeta)$ which, for smooth input $\phi$, is simply the determinant of the complex Hessian of $\phi$.

Suppose now that $\Omega$ is a Reinhardt domain i.e. if $z=(z_1,\cdots, z_n)\in\Omega$, then\newline $(e^{i\theta_1}z_1,\cdots,e^{i\theta_n}z_n)\in \Omega$ for any $\theta_j\in[0,2\pi),\ j=1,\cdots,n$.

Recall the following classical fact (see also \cite{AD21}):
\begin{lem}\label{toric}
	Let $\Omega$ be a Reinhardt domain in $\mathbb C^n$ and $\phi$ be a bounded plurisubharmonic function on $\Omega$, invariant with respect to the toric action. Then:
	\begin{enumerate}
		\item The image $U$ of the mapping 
		$$\Omega\ni z\rightarrow (log|z_1|,\cdots,log|z_n|)\in\mathbb R^n$$
		is a domain in $\mathbb R^n$. $\Omega$ is pseudoconvex if and only if $U$ is convex.
		\item The function $u(x):=\phi(e^{x_1},\cdots,e^{x_n})$ is a convex function in $U$. Reversely for any bounded convex function $u$ on $U$ the function $\phi$ defined through this formula extends to a toric
		plurisubharmonic function on $\Omega$.
		\item $(dd^c\phi)^n=0$ if and only if $det(D^2u)=0$- the equivalence continues to hold for bounded singular $u$ and $\phi$ and then the equalities are understood in weak sense
		of measures- see \cite{Gu02,Kol05}.
	\end{enumerate} 
\end{lem}
Finally we recall that a Reinhardt domain $\Omega$ is called {\it complete} if for any $z=(z_1,\cdots,z_n)\in\Omega$ and for any $|\lambda_j|\leq 1, \lambda_j\in\C, j=1,\cdots,n$ one has $(\lambda_1z_1,\cdots,\lambda_nz_n)\in\Omega$.

{\bf AM-HM matrix inequalities.}

It is a folklore, though rather unknown fact, that the classical inequality between arithmetic and harmonic means holds also for positive definite matrices (proofs of this can be found for example in \cite{BS78} or \cite{ST94}).

\begin{prop}\label{AMHM}
Let $A, B$ be positive definite matrices of the same dimension. Then for any $t\in[0,1]$ the inequality
$$((1-t)A^{-1}+tB^{-1})^{-1}\leq (1-t)A+tB$$
holds in the sense that the difference is a semi-positive definite matrix.
\end{prop}
\begin{proof}
	We sketch an argument, based on the one in \cite{BS78}, for the sake of completeness.
	
	Obviously for any (not necessarily positive definite) symmetric matrix $H$ one has
	$$HAH\geq 0,\ HBH\geq0.$$
	Thus
	$$(1-t)\big[(1-t)A^{-1}+tB^{-1}-A^{-1}\big]A\big[(1-t)A^{-1}+tB^{-1}-A^{-1}\big]$$
	$$+t\big[(1-t)A^{-1}+tB^{-1}-B^{-1}\big]B\big[(1-t)A^{-1}+tB^{-1}-B^{-1}\big]\geq 0.$$
On the other hand the left hand side can be rewritten as
$$\big[(1-t)A^{-1}+tB^{-1}\big]\big((1-t)A+tB\big)\big[(1-t)A^{-1}+tB^{-1}\big]$$
$$+(1-2(1-t)-2t)\big[(1-t)A^{-1}+tB^{-1}\big]$$
$$=\big[(1-t)A^{-1}+tB^{-1}\big]\big((1-t)A+tB-[(1-t)A^{-1}+tB^{-1}]^{-1}\big)\big[(1-t)A^{-1}+tB^{-1}\big].$$
Multiplying from the left and from the right by the inverse of the (positive definite) matrix $(1-t)A^{-1}+tB^{-1}$ we finish the proof.
\end{proof}

We shall also need a weighted version of the above inequality. 
\begin{prop}\label{AMHMweight}
	Let $A=(a_{mk})$ and $B=(b_{mk})$ be as above and let $l_j,\ r_j\ j=1,\cdots,n$ be any positive numbers. Consider the positive definite matrices $\tilde{A}:=(a_{mk}l_ml_k)$ and $\tilde{B}:=(b_{mk}r_mr_k)$.  Then we have
	$$((1-t)(\tilde{A})^{-1}+t(\tilde{B})^{-1})^{-1}\leq \widetilde{(1-t)A+tB}$$
	with $\widetilde{(1-t)A+tB}_{mk}=[{(1-t)A+tB}]_{mk}s_ms_k$, where $s_i=(\frac{1-t}{l_i}+\frac{t}{r_i})^{-1}$ is the harmonic mean of $l_i$ and $r_i$.
\end{prop}

As we were unable to find this exact result in the literature we provide a detailed proof. First we need a lemma that is modeled on Lemma 5.12 from \cite{HP14}.
\begin{lem}
Let $A,\ B$ and $t$ be as above. Let also $<,>$ denote the scalar product in $\mathbb R^n$. Then for any vector $z\in\mathbb R^n$ one has
\begin{equation}\label{hmdef}
<z,[(1-t)A^{-1}+tB^{-1}]^{-1}z>=inf_{x+y=z}\frac{<x,Ax>}{1-t}+\frac{<y,By>}{t}.
\end{equation}
\end{lem}
\begin{proof}
Note that, just as in \cite{HP14}, $[(1-t)A^{-1}+tB^{-1}]^{-1}=A[tA+(1-t)B]^{-1}B$.

Then, for fixed $x$ and $y$ summing to $z$, the difference of the right hand side and the left hand side of \ref{hmdef} reads

\begin{align*}
&\frac{<x,Ax>}{1-t}+\frac{<z-x,B(z-x)>}{t}\\
&-\frac1t<z,[(tA+(1-t)B)-(1-t)B][tA+(1-t)B]^{-1}Bz>\\
&=\frac{<x,Ax>}{1-t}+\frac{<z,Bz>}{t}+\frac{<x,Bx>}{t}-2\frac{<z,Bx>}{t}\\
&-\frac{<z,Bz>}{t}+\frac{1-t}{t}<Bz,[tA+(1-t)B]^{-1}Bz>\\
&=\frac{<x,[tA+(1-t)B]x>}{t(1-t)}+\frac{1-t}{t}<Bz,[tA+(1-t)B]^{-1}Bz>-2\frac{<Bz,x>}{t}\\
&=\frac1{t(1-t)}||[tA+(1-t)B]^{1/2}x-(1-t)[tA+(1-t)B]^{-1/2}Bz||^2
\end{align*}
with equality only for $x=(1-t)[tA+(1-t)B]^{-1}Bz$.
\end{proof}
\begin{proof}(Proof of the weighted AM-HM inequality)
	
	Of course it suffices to show that for any $z\in\mathbb R^n$ one has
	$$<z,((1-t)(\tilde{A})^{-1}+t(\tilde{B})^{-1})^{-1}z>\leq <z,[\widetilde{(1-t)A+tB}]z>$$
	
	In view of the previous lemma the left hand side is dominated by
	 $$\frac{<x,\tilde{A}x>}{1-t}+\frac{<y,\tilde{B}y>}{t}$$
for any two vectors $x$ and $y$ summing up to $z$. Choose now $x=(x_1,\cdots,x_n),\ y=(y_1,\cdots,y_n)$ by
$x_i=\frac{1-t}{l_i}s_iz_i$ and $y_i=\frac{t}{r_i}s_iz_i$ (by the definition of $s_i$ we have $x+y=z$). Then	
\begin{align*}
&\frac{<x,\tilde{A}x>}{1-t}+\frac{<y,\tilde{B}y>}{t}\\
&\sum_{i,j=1}^n[(1-t)a_{ij}s_iz_is_jz_j+tb_{ij}s_iz_is_jz_j]=<z,[\widetilde{(1-t)A+tB}]z>,
\end{align*}
as claimed.
\end{proof}
\section{Examples}

We begin with the following example borrowed from \cite{AD21}:
\begin{exa}\label{realexample} We take \;$U =(-1,1)$, $u_0(x)=2(x^2-1)$, $u_1(x)=x^2-1$. Then the weak geodesic $u(x,t)$ is given by the formula
	\begin{equation*}u(x,t):=
		\begin{cases}
			2(1-t)\left((\frac{x+t}{1-t})^2-1\right)& \text{if}\;\frac{x+t}{1-t}<\frac{-1}{2};\\
			\frac{2x^2}{1+t}+t-2 & \text{if}\; \frac{x+t}{1-t}\geq\frac{-1}{2}\; \text{and}\; \frac{x-t}{1-t}\leq \frac{1}{2};\\
			2(1-t)\left((\frac{x-t}{1-t})^2-1\right)& \text{if} \;\frac{x-t}{1-t}\geq\frac{1}{2}.
		\end{cases}
	\end{equation*}
\end{exa}
Observe that $\partial u_0(U)\neq \partial u_1(U)$. Note also that both $u_0,\ u_1$ belong to $\tilde{\mathcal S}$ as well as to $\tilde{\mathcal S}^{1,\alpha}$ for any $\alpha\in(0,1]$. The solution however fails to be better than Lipschitz globally as $D_xu=-2$ along $x=-\frac{t+1}2$, while $D_xu=-4$ along $x=-1$.

Thus if one seeks better global regularity the condition $\partial u_0(U)= \partial u_1(U)$ is necessary, while for local regularity one has to work away form the corner set $\partial U\times\{0,1\}$.

The second example shows that in $\tilde{\mathcal S}$, assuming $\partial u_0(U)= \partial u_1(U)$, no blow-up of the second order derivatives occurs despite the failure of the strict convexity of $u_0,\ u_1$ at the boundary of $U$:
\begin{exa}\label{e1}
	Let $U=(-1,1)$ and the functions $u_0,\ u_1$ are given by 
	$$u_0(x)=\log(x^2+1)-\log(2)\text{,}\;u_1(x)=-\frac{1}{8}x^4+\frac{3}{4}x^2-\frac{5}{8}.$$
	Then the weak geodesic $u$ joining them is globally $C^{1,1}$. 
	\end{exa}
Note that both $u_0,\ u_1$ belong to $\tilde{\mathcal S}\setminus \mathcal S$.
	It is straightforward to check to check that $u_0'(-1,1)= u_1'(-1,1)=(-1,1)$ i.e. the gradient images agree. 
	
	Furthermore,  $u_0$ and $u_1$ satisfy the following properties:$$u_0(\pm1)= u_1(\pm 1)=0,\; 
	u_0^{''}(x)=\frac{2(1-x^2)}{(x^2+1)^2}\;\text{and}\; u_1^{''}(x)=\frac{3}{2}(1-x^2).$$
	In order to compute the geodesic $u$ joining them we shall exploit the formula (\ref{fla}). 
	
	The supremum in the Legendre transformation of $u_0$ is achieved at $x_0=x_0(s)=\frac{1-\sqrt{1-s^2}}{s}$ solving $s=\frac{2x}{x^2+1}$. Hence the transform reads
	$$
	u^{*}_0(s)=1-\sqrt{1-s^2}-\log(1-\sqrt{1-s^2})+\log(s^2)=sx_0(s)-u_0(x_0(s)).$$
	
	In the case of $u_1$ the supremum is achieved at $x_1=x_1(s)$ given implicitly by $s=-\frac12x^3+\frac32x$.
	By solving the cubic equation one has
	\begin{eqnarray*}
		u_1^{*}(s)&=& 2s  \cos\left(\frac{1}{3}\arccos(-s)-\frac{2\pi}{3}\right)+2 \cos^4\left(\frac{1}{3}\arccos(-s)-\frac{2\pi}{3}\right)\\
		&-& 3\cos^2\left(\frac{1}{3}\arccos(-s)
		-\frac{2\pi}{3}\right)+\frac{5}{8}\\
		&=&sx_1(s)-u_1(x_1(s)).
	\end{eqnarray*}
	Denote now by $u_t^*(s)$ the partial Legendre transform of $u=u(x,t)$ in the $x$-direction.

	By Proposition (\ref{Ras}) we have,
	\begin{eqnarray*}
		&u^{*}_t(s)= (1-t)u^{*}_0(s)+tu^{*} _1(s)\\
		&=(1-t)\left( 1-\sqrt{1-s^2}-\log(1-\sqrt{1-s^2})+\log(s^2)\right)+t\biggl( 2s  \cos\left(\frac{1}{3}\arccos(-s)-\frac{2\pi}{3}\right)\\
		&+ 2 \cos^4\left(\frac{1}{3}\arccos(-s)-\frac{2\pi}{3}\right)
		- 3\cos^2\left(\frac{1}{3}\arccos(-s)
		-\frac{2\pi}{3}\right)+\frac{5}{8}\biggl).
	\end{eqnarray*}

Before we continue we note the implicit relations

\begin{equation}\label{star}
	s=u_0'(x_0(s)),\ \ s=u_1'(x_1(s)).
\end{equation}

By Proposition \ref{propLag} $u(x,t)=u_t(x)=(u_t^*)^*$. The extremum in the definition of the Legendre transform is reached for
$$y=\frac{\partial}{\partial s}u_t^*(s)$$
which, exploiting (\ref{star}) is easily seen to give
\begin{equation}\label{y}
y=(1-t)x_0(s(y,t))+tx_1(s(y,t)),
\end{equation}
where $s(y,t)$ is an implicit function defined through (\ref{y}).

Thus finally
\begin{equation}\label{uyt}
	u(y,t)=(1-t)u_0(x_0(s(y,t)))+tu_1(x_1(s(y,t))).
\end{equation}

Hence we compute, exploiting (\ref{uyt}), (\ref{star}) and (\ref{y})
\begin{align*}
	&u_y'(y,t)=(1-t)u_0'(x_0(s(y,t)))x_0'(s(y,t))s'_y(y,t)+tu_1'(x_1(s(y,t)))x_1'(s(y,t))s'_y(y,t)\\
	&=s(y,t)\big[(1-t)x_0'(s(y,t))s'_y(y,t)+tx_1'(s(y,t))s'_y(y,t)\big]\\
	&=s(y,t).
\end{align*}
In order to obtain $u_{yy}''$ and $u_{yt}''$ we differentiate the last equality in $y$ and $t$ direction respectively. Using (\ref{y}) and (\ref{star}) once again we obtain

\begin{equation}\label{flayy}
u_{yy}''(y,t)=(\frac{1-t}{u_0''(x_0(s(y,t)))}+\frac{t}{u_1''(x_1(s(y,t)))})^{-1}=:\frac 1H,	
\end{equation}
\begin{equation}\label{flayt}
	u_{yt}''(y,t)=\frac{x_0(s(y,t))-x_1(s(y,t))}{H}.\end{equation}
Finally from the homogeneous Monge-Amp\`ere equation we obtain
\begin{equation}\label{flatt}
	u_{tt}''(y,t)=\frac{(x_0(s(y,t))-x_1(s(y,t)))^2}{H}.
\end{equation}

As $x_0,\ x_1$ are uniformly bounded the Hessian would be under control provided we have a bound of $\frac 1H$ from above. But from the arithmetic and harmonic means inequality
$$0< \frac1H\leq (1-t)u_0''(x_0(s(y,t)))+tu_1''(x_1(s(y,t)))\leq C,$$
for a uniform $C$ despite both $u_0$ and $u_1$ have vanishing second derivatives at the boundary. 

Our last example covers the situation when $\varphi_0$ belongs to $\mathcal H$ but $\varphi_1$ fails to be strictly plurisubharmonic in $\Omega$.

\begin{exa}\label{3rdexa}
Let $\Omega$ be the unit disc in $\mathbb C$. Let also $\varphi_0(z)=2(|z|^2-1),\ \varphi_1(z)=|z|^4-1$. Then the geodesic $\phi$ joining them is $C^{1,1}$ (even $C^2$) smooth, despite the fact that $\varphi_1$ fails to be strictly plurisubharmonic at zero.
\end{exa}

Note that logarithmic transform sends the unit disc to $\mathbb R_{<0}$. then 
$u_0(x)=\varphi_0(e^x)=2(e^{2x}-1)$ while $u_1(x)=e^{4x}-1$. The constants are chosen so that the gradient images of $u_0$ and $u_1$ agree.

If $(\xi,0)$ and $(\eta,1)$ are the endpoints of the segment passing through $(x,t)$ along which the real geodesic $u(x,t)=\phi(e^x,e^t)$ is affine, then a direct computation shows that
$$\xi(x,t)=2\eta(x,t),\ \ \ \eta(x,t)=\frac{2x}{2-t}.$$
Thus
\begin{align*}
&u(x,t)=(1-t)u_0(\xi(x,t))+tu_1(\eta(x,t))\\
&=(2-t)[e^{\frac{4x}{2-t}}-1]
\end{align*}
and hence
\begin{equation}\label{fla3rdexa}
	\phi(z,\zeta)=(2-log|\zeta|)[|z|^{\frac{4}{2-log|\zeta|}}-1].
\end{equation}
We remark that in this example the difference $\xi-\eta$ is unbounded as $x\rightarrow-\infty$.

\section{Proof of the main theorem}

Below we provide a proof of Theorem \ref{main}. Not surprisingly the proof will follow the corresponding argument from \cite{AD21} up to a point where an important amendment is made. For this reason we shall keep the notation from there.

The first lemma, just as in \cite{AD21}, is taken from \cite{CNS86}.

\begin{lem}\label{lem1} Let $(x^0,t^0)$ be any point in $U\times(0,1)$. Subtracting a linear function if necessary, we may suppose that 
	$$ u\geq 0,\;\;\; u(x^0,t^0)=0.$$
	Then $(x^0,t^0)$ is in the convex hull of $(n+1)$ points (not necessarily distinct)\newline $(x^1,t^1), (x^2,t^2),...(x^{n+1},t^{n+1})$ in $\partial (U\times (0,1))$ with $u(x^i,t^i)=0$ for all $i\in \{1,2,...n+1\}.$
\end{lem}

Next we borrow a crucial observation from \cite{AD21}- see the Claim 1 in the proof of inequality (2.1) there:

\begin{lem}\label{AbjDin}
	Such a convex hull consists of a line segment joining a point from $\overline{U}\times\{0\}$ with a point $\overline{U}\times\{1\}$ which passes through $(x^0,t^0)$.
\end{lem}
\begin{proof}
	The argument from \cite{AD21} applies verbatim. Notice that no regularity of the boundary data besides strict convexity (in the sense of the current note) is used there. Note also that the same conclusion is being drawn in formula (\ref{uyt}) in Example \ref{e1}.
\end{proof}

Just as in \cite{AD21} we denote by $\xi(x,t)$ and $\eta(x,t)$ the points in $\overline{U}$, so that $(\xi(x,t),0)$ and $(\eta(x,t),1)$ are the endpoints of the segment containing $(x,t)$ along which $u$ vanishes.

Then Lemma 1.7 from \cite{AD21} states that $\xi$ and $\eta$ are smooth functions in $U\times(0,1)$ for $\varphi_0,\ \varphi_1\in\mathcal S$. This is also true in $\tilde{\mathcal S}$ but since implicit function theorem was used in the argument (see the proof of Lemma 1.7 in \cite{AD21})  a precise control on norms is needed to justify that no blow-up near the boundary occurs. This was the exact reason for assuming strict convexity up to the boundary in \cite{AD21}.

In the following lemma we provide a crucial formula for the Hessian of $u$ in the spatial direction. 
\begin{lem}\label{thirdlemma}Let $u(x,t)$ be a weak geodesic between $\varphi_0$ and $\varphi_1$ in $\tilde{\mathcal{S}}$. We assume furthermore that $\partial \varphi_0(U)=\partial \varphi_1(U)$. Then the second derivative of $u(x,t)$ in the space $x$ direction is  given as follows:
	$$D^2_{xx}u(x,t)=\left((1-t)(D^2_{xx} \varphi_0(\xi))^{-1}+t(D^2_{xx} \varphi_1(\eta))^{-1}\right)^{-1},$$
	where $D^2_{xx}$ is the Hessian of $u$ with respect the variable $x$ only.
\end{lem}
\begin{proof}
Note that the claimed result simply states that formula (\ref{flayy}) from the computations in Example \ref{e1} holds regardless of the space dimension and boundary data.

Fix a point $(x,{t})\in U\times(0,1)$. Linearity of $u$ along $L$ forces the
equality
\begin{equation}\label{firstequalitygradients}
	D_x\varphi_0(\xi({x},{t}))=D_xu(\xi({x},{t}),0)=D_xu(x,t)=D_xu(\eta({x},{t}),1)=D_x\varphi_1(\eta({x},{t})).
\end{equation}

Differentiating this once again one obtains
\begin{equation}\label{flanew}
D^2_{xx}u(x,t)=D^2_{xx}\varphi_0(\xi)D_x\xi=D^2_{xx}\varphi_1(\eta)D_x\eta.
\end{equation}
From the definition of $\tilde{\mathcal S}$ both $D^2_{xx}\varphi_0$ and $D^2_{xx}\varphi_1$ are positive definite and hence invertible.

Next we differentiate in $x$ the equality  
\begin{equation}\label{txieta}
(1-t)\xi(x,t)+t\eta(x,t)=x
\end{equation}
and obtain 
\begin{equation}\label{fladxx}
	Id=(1-t)D_x\xi+tD_x\eta
	\end{equation}
$$=(1-t)(D^2_{xx}\varphi_0(\xi))^{-1}D^2_{xx}\varphi_0(\xi)D_x\xi+t(D^2_{xx}\varphi_1(\eta))^{-1}D^2_{xx}\varphi_1(\eta)D_x\eta$$
$$=\big[(1-t)(D^2_{xx}\varphi_0(\xi))^{-1}+t(D^2_{xx}\varphi_1(\eta))^{-1}\big]\times D^2_{xx}u(x,t),$$
where we applied (\ref{flanew}) to obtain the last equality.

The claimed formula is a direct consequence of (\ref{fladxx}).

\end{proof}
Next we compute the second derivatives in mixed time-spatial direction. Not surprisingly the result generalizes formula (\ref{flayt}).

\begin{lem}\label{xt}
	Let $(u_t)=u(x,t)$  be a geodesic path connecting $\varphi_0$ and $\varphi_1$ in  $\tilde{\mathcal{S}}$. If $$\partial \varphi_0(U)=\partial \varphi_1(U)$$
	then
	$$D^2_{xt}u(x,t)=\big[(1-t)(D^2\varphi_0)^{-1}(\xi)+t(D^2\varphi_1)^{-1}(\eta))\big]^{-1}(\xi-\eta).$$ 
\end{lem}
\begin{proof}
	Differentiating (\ref{firstequalitygradients}) in $t$ direction we obtain
	\begin{equation}\label{ae}
	D_{xt}^2u(x,t)=D_t(D_x\varphi_0(\xi(x,t)))=D^2_{xx}\varphi_0(\xi(x,t))D_t\xi=D^2_{xx}\varphi_1(\eta(x,t))D_t\eta.
	\end{equation}

Differentiating in $t$ the formula (\ref{txieta}), similarly to the previous lemma we obtain
$$0=\eta-\xi+(1-t)D_t\xi+tD_t\eta$$
$$=\eta-\xi+(1-t)(D^2_{xx}\varphi_0(\xi))^{-1}D^2_{xx}\varphi_0(\xi)D_t\xi+t(D^2_{xx}\varphi_1(\eta))^{-1}D^2_{xx}\varphi_1(\eta)D_x\eta$$
$$=\eta-\xi+\big[(1-t)(D^2\varphi_0)^{-1}(\xi)+t(D^2\varphi_1)^{-1}(\eta))\big]\times D^2_{xt}u(x,t),$$
where (\ref{ae}) was used in the last equality. This proves the lemma.

\end{proof}

Finally, coupling Lemma \ref{thirdlemma}, Lemma \ref{xt} and the geodesic equation we obtain the formula for the second derivative in time direction.
\begin{cor}\label{tt}
		Let $(u_t)=u(x,t)$  be a geodesic path connecting $\varphi_0$ and $\varphi_1$ in  $\tilde{\mathcal{S}}$. If $$\partial \varphi_0(U)=\partial \varphi_1(U)$$
	then
	$$D^2_{tt}u(x,t)=(\xi-\eta)^T\big[(1-t)(D^2\varphi_0)^{-1}(\xi)+t(D^2\varphi_1)^{-1}(\eta)\big]^{-1}(\xi-\eta).$$
\end{cor}

We have gathered all the ingredients to conclude the proof of Theorem \ref{main}:
\begin{proof}
	If $\partial\varphi_0(U)\neq \partial\varphi_0(U)$ then $u$ is well-known not to be globally $C^{1,1}$ regular - see Theorem 0.3 in \cite{AD21}. The local $C^{1,1}$ regularity away from the corners will be shown in the next section (just take $\alpha=1$ there). Henceforth we assume $\partial\varphi_0(U)=\partial\varphi_0(U)$.
	
 In	Lemma \ref{thirdlemma}, Lemma \ref{xt} and Corollary \ref{tt} we have computed the Hessian of $u(x,t)$. As $(\xi-\eta)$ is a bounded vector with length bounded by $diam(U)$ the Hessian will be uniformly bounded provided an upper control of $\big[(1-t)(D^2\varphi_0)^{-1}(\xi)+t(D^2\varphi_1)^{-1}(\eta)\big]^{-1}.$  But Proposition \ref{AMHM} implies that the above matrix is majorized by
 $$(1-t)(D^2\varphi_0)(\xi)+t(D^2\varphi_1)(\eta)$$
 and the latter is obviously bounded from the assumptions made on $\varphi_0$ and $\varphi_1$.
\end{proof}

\section{$C^{1,\alpha}$ regularity}
In this section we provide the proof of Theorem \ref{thm2}. The main difference from the corresponding result in \cite{AD21} (see Theorem 0.1 there) is that there we have used the refined test for $C^{1,1}$ boundedness from \cite{CNS86} which is well known to fail for $\alpha\in (0,1)$ - see the remark after Lemma 3.2 in \cite{DPF15}.

Henceforth we shall use the following result (see \cite{DPF15}, Lemma 3.1 or \cite{Fi17}, Lemma A.32.)
\begin{lem}\label{Fig}
	Let $U$ be a convex open set, $C$ and $\rho$ be positive constants, and
	$\alpha\in(0,1]$. Let $u: U\longmapsto \mathbb R$ be a globally Lipschitz convex function such that for every
	$y\in U$  there exists a supporting affine function $l_y$ such that
	$$u(x)-l_y(x)\leq C|x-y|^{1+\alpha}$$
	for all $x\in U\cap B_{\rho}(y)$.
	Then $u\in C^{1,\alpha}(U)$.
	
\end{lem}

Fix now two neighborhoods $W_1\Subset W_2$ of $\partial U\times\{0,1\}$, such that both
$$V_i:=U\times(0,1)\setminus \overline{W_i},\ i=1,2$$
are convex. This choice implies that there is $\delta>0$ such that for any $(y,s)\in W_1,\ (x,t)\in V_2$ one has
$$|(y,s)-(x,t)|\geq\delta.$$
With the aid of Lemmas \ref{Fig} and \ref{AbjDin}, Theorem \ref{thm2} follows from the following bound: there is a $\rho>0$ dependent on $U$ and $\delta$, such that for every point $(\overline{x},\overline{t})$ in $V_2$
and any $(x,t)\in U\times(0,1)$ such that 
$|(\overline{x},\overline{t})-(x,t)|<\rho$ 
one has
\begin{equation}\label{2}
	|u(x,t)-u((\overline{x},\overline{t}))-(x-\overline{x}).D_xu((\overline{x},\overline{t}))-(t-\overline{t})D_tu((\overline{x},\overline{t}))|
\end{equation}
\begin{equation*}\leq C|(x-\overline{x},t-\overline{t})|^{1+\alpha}
\end{equation*}
for a constant $C$ depending on $V_2$, $\delta$, $\rho$ and the $C^{1,\alpha}$ norm of $\varphi_0$,$\varphi_1$ (but independent on $(x,s)$).

\begin{proof}[Proof of inequality (\ref{2})]
We follow the reasoning from \cite{AD21}. We fix $(\bar{x},\bar{t})\in V_2$. After possibly subtracting an affine support function $l_{\bar{x},\bar{t}}(x,t)$ we can suppose 
	$$ u \geq 0,\;\;  u(\bar{x},\bar{t})=0 \; \text{and}\; D u(\bar{x},\bar{t})=0.$$
	The problem (\ref{ss}) for this new function $u$ becomes:
	\begin{equation}\label{sss}
		\begin{cases}
			\det(D^2 _{x,t}u)=0\;in \;\; U\times (0,1),\\
			u=\varphi_0-l_{\bar{x},\bar{t}}\;\; on \;\; U\times \{0\},\\
			u=\varphi_1-l_{\bar{x},\bar{t}}\;\; on \;\; U\times \{1\},\\
			u=-l_{\bar{x},\bar{t}}\;\; on \;\; \partial U\times (0,1).
		\end{cases}
	\end{equation}
	The inequality we need to prove reads
	\begin{equation}\label{am}
		u(x,t) \leq C|(x-\bar{x},t-\bar{t})|^{1+\alpha}
	\end{equation}
	for all $(x,t)$ in $U\times(0,1)$ which are at distance less than $\rho$  from $(\bar{x},\bar{t})$ for some constant $\rho$ to be determined later on.

	By Lemma \ref{AbjDin} we assume that $(\bar{x},\bar{t})$ is on the line segment $[(x^0,0),(x^1,1)]$ on which $u$ vanishes. Switching the role of the end points, if necessary, that $(\bar{x},\bar{t})$ is closer to $(x^1,1)$ than to $(x^0,0)$ (i.e. we assume $\bar{t}\geq\frac12$).
	
	{\bf Claim}: there is a constant $\rho>0$ with the following property: unless both $(x^0,0),\ (x^1,1)$ belong to $W_1$ then at least one of the rays $[(x^0,0),(x,t))$, $[(x^1,1),(x,t))$ strikes the boundary of $U\times(0,1)$ at $\overline{U}\times\{0,1\}$.
	
	Indeed, if $(x^1,1)\in \overline{U}\times\{1\}\setminus \overline{W_1}$, then choose the ray $[(x^0,0),(x,t))$ and (exploiting the fact that $\hat{t}\geq\frac12$) the existence of $\rho$ follows from continuity argument. If in turn $(x^1,1)\in W_1$ but $(x^1,1)\in \overline{U}\times\{1\}\setminus \overline{W_1}$ the same argument  applies for the ray $[(x^1,1),(x,t))$  except now one uses the fact that $|(\overline{x},\overline{t})-(x,t)|\geq\delta$ instead of $\overline{t}\geq\frac12$.
	
	Consider now a ball in $U\times (0,1)$ with center $(\bar{x},\bar{t})$ and radius $\rho$. For any $(x,t)\in U\times(0,1)$ in the $\rho$- ball let $(\hat{x},\hat{t})$ be the point  where the ray from $(x^0,t^0)$ to $(x,t)$ strikes the boundary of $U\times(0,1)$. Let also $(\tilde{x},\tilde{t})$ be the point  where the ray from $(x^1,t^1)$ to $(x,t)$ strikes the boundary of $U\times(0,1)$
	
	Three cases may occur.

	Suppose first that $\hat{t}=1$ i.e. the point $(\hat{x},\hat{t})$ is on $\overline{U}\times\{1\}$. In this case the argument from \cite{AD21} applies mutatis mutandis. We recall the details for the sake of completeness. 
	
	 If
	$(x,t)=t(\hat{x},1)+(1-t)(x^1,0)$ by convexity of $u$ and $u(x^0,0)=0$, we have 
	$$u(x,t) \leq t\tilde{\varphi}_1(\hat{x})+(1-t)u(x^0,0)\leq \tilde{\varphi}_1(\hat{x}),$$
	where $\tilde{\varphi}_1=\varphi_1-l_{\bar{x},\bar{t}}$. Note that $x^1$ is a local minimum point for $\tilde{\varphi}$. As furthermore $\varphi_1\in\tilde{\mathcal S}^{1,\alpha}$ we have
	$$ \tilde{\varphi}_1(\hat{x})-\tilde{\varphi}_1(x^1)\leq C|\hat{x}-x^1|^{1+\alpha}$$
	for a constant $C$ dependent on the $C^{1,\alpha}$ norm of $\varphi_1$ and the global Lipschitz norm of $u$.

	It thus suffices to  prove that 
	\begin{equation}\label{needed}
		|\hat{x}-x^1|\leq C(|(x-\bar{x},t-\bar{t})|,
	\end{equation}
	for some $C$. To this end we consider the plane $\Pi$ spanned by $(x^0,0),(x^1,1)$ and $(x,t)$ (note that if these three points are co-linear then the estimate is trivial). In $\Pi\cap\big[ U\times(0,1)\big]$  we have 
	$$\frac{|\hat{x}-x^1|}{|(x^0,t^0)-(x^1,t^1)|}=\frac{\sin(\beta)}{\sin(\alpha\pm\beta)}
	$$
	and 
	$$\frac{|(\bar{x},\bar{t})-(x^0,t^0)|}{|(x,t)-(\bar{x},\bar{t})|}=\frac{\sin(\theta)}{\sin(\beta)},$$
	where $\alpha$  is the angle between the line $((x^1,1),(x^0,0))$ and the line $\{t=1\}\cap\Pi$, $\beta$ is the angle between $((x^0,0),(x^1,1))$ and $((x^0,0),(\hat{x},1))$
	and $\theta$ is the angle between $((\bar{x},\bar{t}),(x,t))$ and $((x^0,0),(\hat{x},1))$. From these two equations we obtain 
	$$\frac{|\hat{x}-x^1|}{|(x,t)-(\bar{x},\bar{t})|}=\frac{\sin(\theta)}
	{\sin(\alpha\pm\beta)}\frac{|(x^0,0)-(x^1,1)|}{|(\bar{x},\bar{t})-(x^0,0)|}.$$
	
	Note that the angle $\alpha\pm\beta$ is uniformly bounded from below by a constant dependent on $diam(U)$. Also, trivially, $sin(\theta)\leq 1$.
	
	In order to bound the ratio $$\frac{|(x^0,0)-(x^1,1)|}{|(\bar{x},\bar{t})-(x^0,0)|}$$ just observe that $|(x^0,0)-(x^1,1)|\leq\sqrt{diam(U)^2+1}$, while $|(\bar{x},\bar{t})-(x^0,0)|\geq|\bar{t}|\geq\frac12.$ Thus the estimate (\ref{needed}) is proven in this case.
	
	Assume now $\hat{t}\neq1$ which implies that $(x^1,1)\in W_1$.
	
	If, in turn, $\tilde{t}=0$ i.e. $(\tilde{x},\tilde{t})\in\overline U\times\{0\}$ we apply similar idea. We need an analogue of (\ref{needed}) i.e.
	\begin{equation}\label{needed2}
		|\tilde{x}-x^0|\leq C(|(x-\bar{x},t-\bar{t})|.
	\end{equation} 
Arguing as above this boils down to establishing a bound from above on $\frac{|(x^0,0)-(x^1,1)|}{|(\bar{x},\bar{t})-(x^1,1)|}$. Once again $|(x^0,0)-(x^1,1)|\leq\sqrt{diam(U)^2+1}$ but for the lower bound of $|(\bar{x},\bar{t})-(x^1,1)|$ we exploit that now $(x^1,1)\in W_1$ and hence $|(\bar{x},\bar{t})-(x^1,1)|\geq\delta$.

Finally it remains to analyze the case when both $\tilde{t}\neq 0$ and $\hat{t}\neq1$. In this case both points $(x^0,0)$ and $(x^1,1)$ belong to $W_1$. Once again consider the plane $\Pi$ spanned by $(x^0,0),(x^1,1)$ and $(x,t)$ and let $(\overline{x},1), (\underline{x},0)$ be the vertices of the rectangle $\Pi\cap\overline{U\times(0,1)}$ that are close to $(x^1,1)$ and $(x^0,0)$, respectively.
Note that, by convexity
\begin{align*}
	&u(x,t)\leq \frac t{\hat{t}}u(\hat{x},\hat{t})+(1-\frac t{\hat{t}})u(x^0,0)\\
	&= \frac t{\hat{t}}\big( \hat{t}u(\overline{x},1)+(1-\hat{t})u(\underline{x},0)\big),
	\end{align*}
where we have used the linearity of $u$ along $\partial U\times[0,1]$.

As $u(\overline{x},1)\leq C|\overline{x}-x^1|^{1+\alpha}$ and $u(\underline{x},0)\leq C|\underline{x}-x^0|^{1+\alpha}$ it suffices to bound the ratios
$$\frac{|\overline{x}-x^1|}{|(x,t)-(\bar{x},\bar{t})|},\ \frac{|\underline{x}-x^0|}{|(x,t)-(\bar{x},\bar{t})|}.$$	

Note however that if $(p,1)$ is the intersection point of the ray $[(x^0,0),(x,t))$ with $\{t=1\}$, while $(q,0)$ is the intersection point of the ray $[(x^1,1),(x,t))$ with $\{t=0\}$ one obviously has 
$$|\overline{x}-x^1|\leq|p-x^1|,\ |\underline{x}-x^0|\leq |q-1|,$$
whereas the ratios 
$$\frac{|p-x^1|}{|(x,t)-(\bar{x},\bar{t})|},\ \frac{|q-x^0|}{|(x,t)-(\bar{x},\bar{t})|}$$
can be bounded in exactly the same way as in the previous two cases (note that it does not matter that $p$ and $q$ stick out of $\Pi\cap\overline{U\times(0,1)}$). This finishes the proof.

\end{proof}
\begin{rem} Just as in \cite{AD21} the obtained bound is stronger than purely interior one. It implies that the blow-up of the $C^{1,\alpha}$ norm can occur only at the wedge $\partial U\times\{0,1\}$.
\end{rem}
\section{Toric plurisubharmonic setting}
As we already explained toric plurisubharmonic functions can be written as the composition of a convex function with a logarithmic map. This allows us to derive several regularity results concerning geodesics connecting two toric plurisubharmonic functions. 

Recall that given a bounded Reinhardt domain $\Omega\subset \mathbb{C}^n$ we consider the space $\mathcal H$ given by
\begin{equation}\label{defnh}
\mathcal H=\lbrace \phi\in C^{\infty}(\Omega)\cap C(\overline{\Omega})|\ \phi=0\ {\rm on}\ \partial \Omega,||\phi||_{C^2}<\infty,\ D^2_{z\bar{z}}\phi>0\ {\rm}\ in\ \Omega\rbrace,
\end{equation}
with $D^2_{z\bar{z}}\phi$ denoting the complex Hessian $\frac{\partial^2\phi}{\partial z_j\partial{\overline{z}_k}}$.

 From now on we restrict attention to $\tilde{\mathcal H}$ consisting of {\it toric} plurisubharmonic functions in $\mathcal H$. i.e. those for which $\phi(z)=\phi(z')$, whenever $|z_j|=|z_j'|,\ j=1,\cdots,n$. For such functions, as already explained the function $u(x):=\phi(e^{x_1},\cdots,e^{x_n})$ is convex on $U$ - the logarithmic image of $\Omega$.

The next lemma links the plurisubharmonicity of $\phi$ with the convexity of $u$ in a quantitive way.
\begin{lem}\label{convtor}Let $\varphi$ be a toric plurisubharmonic function from $\tilde{\mathcal H}$ and $u$ be as above. Then 
	$$\varphi(z_1,z_2,...,z_n)=u(\log|z_1|,\log|z_2|,,,\log|z_n|)=:u(x),$$
	and hence
	\begin{equation*}
		\frac{\partial^2 \varphi}{\partial z_i\partial \bar{z}_j}(z)=\frac{\partial}{\partial z_i}\left(\frac{\partial u}{\partial x_j}\frac{1}{2\overline{z}_j}\right)\\
		=\frac{\partial^2 u}{\partial x_i\partial x_j }(x)\frac{1}{ 4z_i\overline{z}_j}
	\end{equation*}

\end{lem}

As a corollary of the above lemma note that $\frac{\partial^2 u}{\partial x_i\partial x_j }(x)\frac{1}{ 4z_i\overline{z}_j}$ extends smoothly past \newline $\lbrace z_iz_j=0\rbrace\cap\Omega$. The lemma can also be restated as
\begin{equation}\label{torconv}
\frac{\partial^2 u}{\partial x_i\partial x_j }(x)=\frac{\partial^2 \varphi}{\partial z_i\partial \bar{z}_j}(e^x)4e^{x_i}e^{x_j},
\end{equation}
where we used the suggestive notation $e^x=(e^{x_1},\cdots,e^{x_n})$.

Next proposition is essentially contained in \cite{Ras17}. We refer to \cite{AD21} for more details on the topic.

\begin{prop}\label{geodgeod} Let $\Omega$ be a bounded Reinhardt domain in $\mathbb C^n$ with $U$ - its  logarithmic image. Let also $\varphi$ be a weak geodesic joining $\varphi_0$ and $\varphi_1$ from $\tilde{\mathcal H}$. If $u_0(x):=\varphi_0(e^x),\ u_1(x):=\varphi_1(e^x)$, then 
	$u(x,t):=\varphi(e^x,e^t)$ is a weak geodesic joining $u_0,\ u_1\in \tilde{\mathcal S}$.
\end{prop}

This identification justifies the proof of Proposition \ref{pq} (see also Corollary 0.2 in \cite{AD21}) through the already proven regularity for geodesics in $\tilde{\mathcal S}$ and the logarithmic map. The major problem that occurs in this approach is the obvious degeneracy of the logarithmic map along the set $M=\lbrace z_1\cdots z_n=0\rbrace$. Henceforth we shall assume that $\Omega\cap M$ is not empty for otherwise the regularity analysis is finished. Note that $\Omega\cap M\neq \emptyset$ implies that $U$ is unbounded, which as we shall see below additionally complicates computations.

We are ready to state the main result of this section:

\begin{thm}\label{maintoric}
	Let $\Omega$ be a bounded Reinhardt domain in $\mathbb C^n$. Let also $\phi$ be a weak geodesic joining $\varphi_0$ and $\varphi_1$ from $\tilde{\mathcal H}$. Then $\varphi$ is $C_{loc}^{1,1}$ regular in the spatial directions and the blow-up of the $C^{1,1}$ norm can occur only at the corner $\partial\Omega\times\partial A$. If $\Omega$ is complete Reinhardt then the $C_{loc}^{1,1}$ bound holds in all directions. Furthermore $\varphi$ is globally $C^{1,1}$ smooth (and strictly plurisubharmonic) in the spatial direction  if and only if $\partial u_0(U)=\partial u_1(U)$.
\end{thm}

\begin{rem}
It is very likely that under the assumption $\partial u_0(U)=\partial u_1(U)$ $\phi$ is in fact $C^{\infty}$ smooth and hence a classical geodesic in any complete smoothly bounded strictly pseudoconvex Reinhardt domain. This however requires a thorough higher order analysis across the set $\Omega\cap M$. We hope to address this problem in the future.
\end{rem}
\begin{proof}
	The proof will be divided into several steps.
	
	{\bf Step 1} (Reduction to a neighborhood of $M$). Proposition \ref{pq} provides the required regularity except in a neighborhood of $\Omega\cap M$. Fix now a neighborhood $\Theta$ of $\Omega\cap M$ small enough so that the following holds: whenever $z\in \Theta\cap\Omega\setminus M$ and $x=log|z|$ is its logarithmic image, then the line segment joining $(\xi,0)$ and $(\eta,1)$ along which the geodesic $u(x,t)=\varphi(e^x,e^t)$ is affine has both points $\xi$ and $\eta$ in $U$. 
	
	{\bf Step 2} (Reduction to real points) As $\varphi$ is toric both in space and (complex) time variables the full complex  Hessians at points $(z,\zeta)$ and $(z',\zeta')$ for $|z_j|=|z_j'|,\ |\zeta|=|\zeta'|$ are related by
	$$D^2_{z,\zeta}\varphi(z',\zeta')=U^*D^2_{z,\zeta}\varphi(z,\zeta)U$$
	for some $(n+1)\times(n+1)$ unitary diagonal matrix. For this reason if $z\in\Omega\setminus M$ it is sufficient to bound the complex Hessian on {\it real positive} points i.e. for $(z,\zeta)$ such that $Im(\zeta)=Im(z_j)=0,\ Re(\zeta)>0,\ Re(z_j)>0,\ j=1,\cdots,n$. Thus from now on we assume that $(z,\zeta)$ is a real positive point.
	
	{\bf Step 3} (The spatial Hessian). With the aid of Lemma \ref{convtor} and Proposition \ref{torconv} we compute the spatial Hessian
	\begin{equation}\label{zzhess}
		\big(\frac{\partial^2 \varphi}{\partial z_i\partial \bar{z}_j}(z)\big)=\big(\frac{\partial^2 u}{\partial x_i\partial x_j }(x)\frac{1}{ 4z_i\overline{z}_j}\big)
	\end{equation}
	\begin{align*}
		&=\big[\big((1-t)\big(D^2u_0(\xi)\big)^{-1}+t\big(D^2u_1(\eta)\big)^{-1}\big)^{-1}_{ij}\frac{1}{ 4z_i\overline{z}_j}\big]\\
		&=\big[\big((1-t)\big(\frac{\partial^2 \varphi_0}{\partial z_k\partial \bar{z}_l}(e^{\xi})4e^{\xi_k+\xi_l}\big)^{-1}+t\big(\frac{\partial^2 \varphi_1}{\partial z_k\partial \bar{z}_l}(e^{\eta})4e^{\eta_k+\eta_l}\big)^{-1}\big)^{-1}_{ij}\frac{1}{ 4z_i\overline{z}_j}\big]\\
		&=\big[\big((1-t)\big(\frac{\partial^2 \varphi_0}{\partial z_k\partial \bar{z}_l}(e^{\xi})e^{\xi_k+\xi_l}\big)^{-1}+t\big(\frac{\partial^2 \varphi_1}{\partial z_k\partial \bar{z}_l}(e^{\eta})e^{\eta_k+\eta_l}\big)^{-1}\big)^{-1}_{ij}\frac{1}{ z_i\overline{z}_j}\big]
	\end{align*}
	At this point we recall that by assumption $z_j$'s are positive reals. Also the complex Hessian matrices $\frac{\partial^2 \varphi_0}{\partial z_k\partial \bar{z}_l}(e^{\xi})$ and $\frac{\partial^2 \varphi_1}{\partial z_k\partial \bar{z}_l}(e^{\eta})$ are real and positive definite. An application of Proposition \ref{AMHMweight} implies that the matrix $\big(\frac{\partial^2 \varphi}{\partial z_i\partial \bar{z}_j}(z)\big)$ is bounded from above by the matrix
	$$\big[\big((1-t)\big(\frac{\partial^2 \varphi_0}{\partial z_k\partial \bar{z}_l}(e^{\xi})e^{\xi_k+\xi_l}\big)^{-1}+t\big(\frac{\partial^2 \varphi_1}{\partial z_k\partial \bar{z}_l}(e^{\eta})e^{\eta_k+\eta_l}\big)^{-1}\big)^{-1}_{ij}\frac{1}{ e^{(1-t)\xi_i+t\eta_i}e^{(1-t)\xi_j+t\eta_j}}\big]$$
	$$\leq \big[\big((1-t)\big(\frac{\partial^2 \varphi_0}{\partial z_k\partial \bar{z}_l}(e^{\xi})\big)+t\big(\frac{\partial^2 \varphi_1}{\partial z_k\partial \bar{z}_l}(e^{\eta})\big)\big)_{ij}\frac{(\frac{1-t}{e^{\xi_i}}+\frac{t}{e^{\eta_i}})^{-1}(\frac{1-t}{e^{\xi_j}}+\frac{t}{e^{\eta_j}})^{-1}}{ e^{(1-t)\xi_i+t\eta_i}e^{(1-t)\xi_j+t\eta_j}}\big].$$
	The matrix part $\big((1-t)\big(\frac{\partial^2 \varphi_0}{\partial z_k\partial \bar{z}_l}(e^{\xi})\big)+t\big(\frac{\partial^2 \varphi_1}{\partial z_k\partial \bar{z}_l}(e^{\eta})\big)\big)$ is bounded from above by definition. The weights $\frac{(\frac{1-t}{e^{\xi_i}}+\frac{t}{e^{\eta_i}})^{-1}}{ e^{(1-t)\xi_i+t\eta_i}}$ are also bounded (by $1$) as the classical inequality between geometric and harmonic means shows. As $\big(\frac{\partial^2 \varphi}{\partial z_i\partial \bar{z}_j}(z)\big)$ is itself positive definite this provides a uniform bound for its entries and hence a $C^{1,1}$ bound in the spatial direction.
	
	{\bf Step 4} (The Hessian in time-spatial and time-time directions) The idea is to exploit Lemma \ref{xt} and Corollary \ref{tt}. Observe that the formulas stated there hold as long as the endpoint $\xi$ and $\eta$ do not touch the boundary - an assumption that is guaranteed from the first step.
	
	Note that if $\xi-\eta$ were  a bounded vector, or more generally if $|(\frac{1-t}{e^{\xi_i}}+\frac{t}{e^{\eta_i}})^{-1}(\xi_i-\eta_i)|$ were under control then the bound on the Hessian in the remaining directions is immediate. Note however that Example \ref{3rdexa} shows that such a bound is not unconditional.
	
	Assume henceforth that $\Omega$ is complete Reinhardt. Note that this implies that its logarithmic image $U$ has the following geometric property: for any $x\in U$, any $k\in\lbrace1,\cdots,n\rbrace$ and any $s<0$ the point $(x_1,\cdots, x_{k-1},x_k+s,x_{k+1},\cdots,x_n)$ belongs to $U$.
	
	Recall the equality $D_{x_k}u_0(\xi)=D_{x_k}u_1(\eta)$, which we rewrite as
	\begin{equation}\label{2dertrick}
		\int_{-\infty}^0D^2_{x_kx_k}u_0(\xi_1,\cdots, \xi_{k-1},\xi_k+s,\xi_{k+1},\cdots,\xi_n)ds
	\end{equation}
	$$=\int_{-\infty}^0D^2_{x_kx_k}u_1(\eta_1,\cdots, \eta_{k-1},\eta_k+s,\eta_{k+1},\cdots,\eta_n)ds.$$
	
	Now, due to Lemma \ref{convtor} and equation (\ref{torconv}) the equality can be further rewritten as
	
	\begin{equation}\label{2dertrickbis}
		\int_{-\infty}^0\frac{\partial^2\varphi_0}{\partial z_k\partial{\overline{z}_k}}(e^{\xi_1},\cdots, e^{\xi_{k-1}},e^{\xi_k+s},e^{\xi_{k+1}},\cdots,e^{\xi_n})4e^{2(\xi_k+s)}ds
	\end{equation}
	$$=\int_{-\infty}^0\frac{\partial^2\varphi_1}{\partial z_k\partial{\overline{z}_k}}(e^{\eta_1},\cdots, e^{\eta_{k-1}},e^{\eta_k+s},e^{\eta_{k+1}},\cdots,e^{\eta_n})4e^{2(\eta_k+s)}ds.$$
	
	Recall now that, by Step 1, both $e^\xi$ and $e^\eta$ are in a compact subset of $\Omega$ and hence the complex Hessians of $\varphi_0$ and $\varphi_1$ are uniformly bounded both from above and below along the integration paths. Hence for a constant $C$ dependent on $\varphi_0,\varphi_1$ one has
	
	$$\int_{-\infty}^0\frac{\partial^2\varphi_0}{\partial z_k\partial{\overline{z}_k}}(e^{\xi_1},\cdots, e^{\xi_{k-1}},e^{\xi_k+s},e^{\xi_{k+1}},\cdots,e^{\xi_n})4e^{2(\xi_k+s)}ds\geq C^{-1}\int_{-\infty}^0e^{2(\xi_k+s)}ds$$
	$$= \frac{C^{-1}}2e^{2\xi_k},$$
	while
	$$\int_{-\infty}^0\frac{\partial^2\varphi_1}{\partial z_k\partial{\overline{z}_k}}(e^{\eta_1},\cdots, e^{\eta_{k-1}},e^{\eta_k+s},e^{\eta_{k+1}},\cdots,e^{\eta_n})4e^{2(\eta_k+s)}ds\leq C\int_{-\infty}^0e^{2(\eta_k+s)}ds$$
	$$= \frac{C}2e^{2\eta_k}.$$
	
	Plugging these into (\ref{2dertrickbis}) one obtains
	$$e^{\xi_k}\leq C^2e^{\eta_k}.$$
Reversing the roles of $\varphi_0$ and $\varphi_1$ 	results in
	$$e^{\eta_k}\leq C^2e^{\xi_k}.$$
	This implies that the vector $\xi-\eta$ is uniformly bounded.
	
	{\bf Step 5} (Implications in complete Reinhardt domains) The arguments so far have shown that $\partial u_0(U)=\partial u_1(U)$ implies that the geodesic is globally $C^{1,1}$ (jointly in space and time) for a complete bounded strictly pseudoconvex Reinhardt domain $\Omega$. The computations in Step 3 also show that for a fixed $\zeta\in A$ the slice function $z\longmapsto \phi(z,\zeta)$ has an explicit formula for the Hessian (away from $M$). Note that the matrix part in (\ref{zzhess}) is locally uniformly bounded from below (and the bound depends on the lower bounds for $D^2_{z\bar{z}}\varphi_j,\ j=0,1$ away from $\partial \Omega$). More specifically the matrix part can be bounded by either
	
$$\frac{\big(\frac{\partial^2 \varphi_0}{\partial z_k\partial \bar{z}_l}(e^{\xi})4e^{\xi_k+\xi_l}\big)}{1-t}\ {\rm or}\ \frac{\big(\frac{\partial^2 \varphi_1}{\partial z_k\partial \bar{z}_l}(e^{\eta})4e^{\eta_k+\eta_l}\big)}{t}.$$

The uniform bound for $\xi-\eta$ (see Step 4) implies in turn that the weights
$$\frac{e^{\xi_i}}{ e^{(1-t)\xi_i+t\eta_i}}\ {\rm or}\ \frac{e^{\eta_i}}{ e^{(1-t)\xi_i+t\eta_i}}$$
		are also uniformly bounded from below. This yields that the slice functions are strictly plurisubharmonic.
\end{proof}

We finish the note with the following open question:
\begin{quest}
Let $\Omega$ be a Reinhardt domain which is not complete, but still the logarithmic image is unbounded. Let $\varphi_0,\varphi_1\in\tilde{\mathcal H}$ be such that the corresponding convex functions $u_0$ and $u_1$ have the same gradient image. Is it still true that $\xi-\eta$ is a bounded vector?
\end{quest}


\begin{thebibliography}{CTW17}
		\bibitem[Abj19]{Abj19} Abja, S. Geometry and topology of the space of plurisubharmonic functions. J. Geom. Anal 29 (2019), 510-541.
		\bibitem[AD21]{AD21} Abja, S., Dinew, S.
		Regularity of geodesics in the spaces of convex and plurisubharmonic functions.
		Trans. Amer. Math. Soc. 374 (2021), no.6, 3783-3800.
	\bibitem[BS78]{BS78}	Bhagwat, K. V., Subramanian, R.
		Inequalities between means of positive operators.
		Math. Proc. Cambridge Philos. Soc. 83 (1978), no.3, 393-401.
		\bibitem[CNS86]{CNS86}  Caffarelli, L., Nirenberg, L., Spruck, J. The Dirichlet problem for the degenerate Monge-Amp\`ere equation
	Rev. Mat. Iberoamericana 2 (1986), no. 1-2, 19-27.
	
	\bibitem[CTW22]{CTW22} Caffarelli, L. Tang, L., Wang, X.-J. Global $C^{1,\alpha}$ regularity for Monge-Amp\`ere equation and convex envelope Arch. Ration. Mech. Anal. 244 (2022), no. 1, 127-155.
	\bibitem[DPF15]{DPF15} De Philippis, G., Figalli, A., Optimal regularity of the convex envelope.
	Trans. Amer. Math. Soc. 367 (2015), no.6, 4407-4422.
		\bibitem[D99]{D99} Donaldson, S.K. Symmetric spaces, K\"{a}hler geometry and Hamiltonian dynamics, Vol. 196 of A.M.S. Transl. Ser. 2 (1999), 13-33.
		
		\bibitem[Fi17]{Fi17}  Figalli, A. The Monge-Amp\`ere equation and its applications. Zurich Lectures in Advanced Mathematics. 
		European Mathematical Society (EMS), Z\"urich, (2017), 200 pp.
		
	\bibitem[Gu02]{Gu02} Guti\'errez, C.,	 The Monge-Amp\`ere equation,  Nonlinear Differential Equations Appl., 44
		Birkh\"auser Boston, Inc., Boston, MA, 2001, xii+127 pp.
		\bibitem[HP14]{HP14} Hiai, F., Petz, D.
		Introduction to matrix analysis and applications.
		Universitext
		Springer, Cham; Hindustan Book Agency, New Delhi, 2014. viii+332 pp.
		\bibitem[Kol05]{Kol05}  Kolodziej, S. The complex Monge-Amp\`ere equation and pluripotential theory. Memoirs AMS 178 (2005), no. 840, 64 pp.
		
		\bibitem[LW15]{LW15}  Li, Q.-R., Wang, X.-J. Regularity of the homogeneous Monge-Amp\`ere equation. Discrete Contin. Dyn. Syst. 35 (2015), no. 12, 6069-6084.
		\bibitem[Mab87]{Mab87} Mabuchi, T. Some symplectic geometry on compact K\"ahler manifolds. Osaka J. Math., 24 (1987), 227-252.
		\bibitem[Ra17]{Ra17} Rashkovskii, A. Copolar convexity.
		Ann. Polon. Math. 120 (2017), no.1, 83-95.
		\bibitem[Ras17]{Ras17} Rashkovskii, A.
		Local geodesics for plurisubharmonic functions.
		Math. Z. 287 (2017), no.1-2, 73-83.
		
		\bibitem[ST94]{ST94} Sagae, M., Tanabe, K.
		Upper and lower bounds for the arithmetic-geometric-harmonic means of positive definite matrices.
		Linear and Multilinear Algebra 37 (1994), no.4, 279-282.
		
		\bibitem[Sem92]{Sem92} Semmes, S. Complex Monge-Amp\`ere and symplectic manifolds. Amer. J. Math. (1992), 495-550.
		
		\bibitem[Wal68]{Wal68}  Walsh, J. Continuity of envelopes of plurisubharmonic functions. J. Math. Mech. 18 (1968), 143-148.
		
	\end{thebibliography}
\end{document}